\documentclass[11pt]{amsart}

\usepackage[english]{babel}

\usepackage{times,amsfonts,amsmath,amssymb,dsfont} 
\usepackage{graphicx,a4wide}

\usepackage{cases}

\usepackage{soul}
\setstcolor{red}

\usepackage{color}

\definecolor{gr}{rgb}   {0.,   0.69,   0.23 }
\definecolor{bl}{rgb}   {0.,   0.5,   1. }
\definecolor{mg}{rgb}   {0.85,  0.,    0.85}
\definecolor{or}{rgb}   {0.9,  0.5,   0.}

\definecolor{webred}{rgb}{0.75,0,0}
\definecolor{webgreen}{rgb}{0,0.75,0}

\newtheorem{theorem}{Theorem}[section]
\newtheorem{proposition}[theorem]{Proposition}
\newtheorem{lemma}[theorem]{Lemma}

\newtheorem{example}[theorem]{Example}

\theoremstyle{definition}
\newtheorem{definition}[theorem]{Definition}

\theoremstyle{remark}

\newcommand{\ba}{\begin{array}}
\newcommand{\ea}{\end{array}}

\newcommand{\N}{\mathbb{N}}
\newcommand{\Q}{\mathbb{Q}}
\newcommand{\R}{\mathbb{R}}

\renewcommand{\S}{\mathbb{S}}

\newcommand{\one}{\mathds{1}}

\newcommand{\cO}{\mathcal{O}}

\newcommand{\bel}{\begin{equation} \label}
\newcommand{\ee}{\end{equation}}

\newcommand\supp{\operatorname{supp}}

\def\<{\langle}
\def\>{\rangle}

\newcommand{\fract}[2]{\genfrac{}{}{0pt}{}{\scriptstyle #1}{\scriptstyle #2}}

\begin{document}

\title[Eigenvalues for degenerate wells]{Low-lying eigenvalues of semiclassical Schr\"odinger operator with degenerate wells}

\author[J.-F. Bony]{Jean-Fran\c{c}ois Bony}
\address{IMB, Universit\'e de Bordeaux, UMR 5251, 33405 Talence, France}
\email{bony@math.u-bordeaux.fr}

\author[N.\ Popoff]{Nicolas Popoff}
\address{IMB, Universit\'e de Bordeaux, UMR 5251, 33405 Talence, France}
\email{nicolas.popoff@math.u-bordeaux.fr}


\begin{abstract}
In this article, we consider the semiclassical Schr\"odinger operator $P = - h^{2} \Delta + V$ in $\R^{d}$ with confining non-negative potential $V$ which vanishes, and study its low-lying eigenvalues $\lambda_{k} ( P )$ as $h \to 0$. First, we give a necessary and sufficient criterion upon $V^{-1} ( 0 )$ for $\lambda_{1} ( P ) h^{- 2}$ to be bounded. When $d = 1$ and $V^{-1} ( 0 ) = \{ 0 \}$, we are able to control the eigenvalues $\lambda_{k} ( P )$ for monotonous potentials by a quantity linked to an interval $I_{h}$, determined by an implicit relation involving $V$ and $h$. Next, we consider the case where $V$ has a flat minimum, in the sense that it vanishes to infinite order. We give the asymptotic of the eigenvalues: they behave as the eigenvalues of the Dirichlet Laplacian on $I_{h}$. Our analysis includes an asymptotic of the associated eigenvectors and extends in particular cases to higher dimensions.
\end{abstract}

\maketitle

\section{Introduction}

The paper is devoted to the study of the spectrum of the Schr\"odinger operator 
\begin{equation} \label{a25}
P = - h^{2} \Delta + V ,
\end{equation}
with semiclassical parameter $h > 0$, acting in $L^{2}(\R^{d})$. The potential $V : \R^{d}\to \R_{+}$ is continuous and satisfies $\liminf_{\vert x \vert \to + \infty} V ( x ) > 0$ and $\inf V = \min V = 0$. The operator $P$ is self-adjoint and non-negative. For such an operator $T$, let $\lambda_{k} ( T )$ be its $k$-th eigenvalue or the bottom of its essential spectrum if $T$ has less than $k$ eigenvalues before it. The asymptotic of the eigenvalues of $P$ as $h \to 0$ has received large considerations from basis of quantum mechanics to microlocal analysis. For fixed $k$, they converge to $0$ as $h \to 0$ and their asymptotic behavior depends on $V^{- 1} ( 0 )$ and on the shape of $V$ near this set.

We first consider the general behavior of $\lambda_{1} ( P )$ under weak assumptions. If $V^{-1} ( 0 )$ contains an open set, it is clear from the maximin principle that $\lambda_{1} ( P ) = \cO ( h^{2} )$. But if $V^{- 1} ( 0 )$ has measure $0$, then $\lambda_{1} ( P ) h^{- 2}$ is unbounded, see \cite[Lemma 3.2]{BelHelVer01}. We will give in section \ref{s2} a characterization on $V^{-1} ( 0 )$ for which $\lambda_{1} ( P ) h^{- 2} \to + \infty$ as $h \to 0$: this is true if and only if $V^{- 1} ( 0 )$ is $1$-null, where the notion of $1$-nullity, coming from \cite{AdaHed96,HewMoi17}, is introduced in Definition \ref{a26}.

Numerous works focus on the case where $V^{- 1} ( 0 )$ is reduced to a point, let us say $V^{- 1} ( 0 ) = \{ 0 \}$. In dimension $d = 1$, the most well-known model case is given by the quadratic harmonic oscillator $V ( x ) = x^{2}$, for which $\lambda_{k} ( P ) = ( 2 k - 1 ) h$. Next, the standard harmonic approximation states that if $V$ is $C^{\infty} ( \R )$ and if $V$ is non-degenerate at $0$, in the sense that $V^{\prime \prime} ( 0 ) > 0$, then we have
\begin{equation*}
\lambda_{k} ( P ) =\sqrt{\frac{V^{\prime \prime} ( 0 )}{2}} ( 2 k - 1 ) h + \cO ( h^{3 / 2} ) ,
\end{equation*}
as $h \to 0$ (see \cite{Si82,HeSj84,CyFrKiSi87,DiSj99}). The idea behind this result is that one can replace the potential $V$ by its Taylor expansion near $0$. Following this strategy, the case where $V$ vanishes to higher order is treated in \cite{MarRou88}. The authors assume that the potential admits a non-zero Taylor expansion $V ( x ) = x^{2 p} + \cO ( x^{2 p + 1} )$ near 0 with $p \geq 2$ (see also \cite{Vo83} for $p = 2$). In that case,
\begin{equation*}
\lambda_{k} ( P ) \sim h^{\frac{2 p}{p + 1}} \lambda_{k} ( - \Delta + x^{2 p} ) .
\end{equation*}
All these results have extensions to higher dimensions. The strategy of the proofs of the previous results does not seem to adapt for more general potentials, in particular for those which have no homogeneous leading term near the minimum.

In this article we describe the low-lying eigenvalues of $P$ without the assumption that $P$ has an expansion near 0. In section \ref{s3} we only assume that $V$ is monotonous near $0$ and we give a control from below and from above for $\lambda_{k} ( P )$ by $h^{2} \vert I_{h} \vert^{- 2}$, where $I_{h}$ is the small interval around $0$ defined by the implicit relation $I_{h} = V^{- 1} ( ( 0 , h^{2}  \vert I_{h} \vert^{- 2} ) )$. Roughly speaking, this interval equilibrates the kinetic energy and the potential energy of the operator, in the sense that 
\begin{equation*}
\frac{h^{2}}{\vert I_{h} \vert^{2}} \approx \lambda_{1} ( - h^{2} \Delta^{D}_{I_{h}} ) \approx \sup_{I_{h}} ( V ) ,
\end{equation*}
where $- \Delta^{D}_{I}$ denotes the Dirichlet Laplacian on the interval $I$.

Next, we give the asymptotic of $\lambda_{k} ( P )$ in the extremal case where $V$ is flat at $0$, i.e. $V ( x ) = \cO ( \vert x \vert^{n} )$ near $0$ for all $n \in \N$. More precisely, we obtain in dimension $1$ that
\begin{equation*}
\lambda_{k} ( P ) \sim \pi^{2} k^{2} h^{2} \vert I_{h} \vert^{- 2} ,
\end{equation*}
under some assumptions stated in Section \ref{s4}. Our proof is based on the fact that $V$ can be replaced by $0$ on the interval $I_{h}$ and is large outside. The asymptotic follows from known results on Schr\"odinger operators with large coupling constant. We provide examples such as $V ( x ) = e^{- \vert x \vert^{- \alpha}}$ with an estimate of remainder (see \eqref{a30}). We also give the asymptotic of the eigenfunctions in Proposition \ref{a20}. Finally, we explain in section \ref{s5} how to adapt our method in higher dimensions for potentials of the form $V ( x ) = V_{0} ( \vert x \vert \theta ( x \vert x \vert^{- 1} ) )$, where $V_{0}$ is a flat potential and $\theta$ is a continuous function on $\S^{d - 1}$.

\section{General estimates} \label{s2}

In this section we give information on the behavior of $\lambda_{1} (P) h^{- 2}$ as $h \to 0$. First, this ratio is bounded from below, since
\begin{equation}
\forall h\in (0,1), \qquad \lambda_{1} ( P ) \geq h^{2} \lambda_{1} ( - \Delta + V ) ,
\end{equation}
with $\lambda_{1} ( - \Delta + V ) > 0$ by unique continuation. Next, we give a necessary and sufficient criterion for its boundedness, in link with $V^{- 1} ( 0 )$. Note that, $h \mapsto \lambda_{1} ( P ) h^{- 2}$ being decreasing on $( 0 , + \infty )$, either this function is bounded near $0$, or it goes to $+ \infty$. It is known that this latter case happens when the Lebesgue measure of $V^{- 1} ( 0 )$ is zero, see \cite[Lemma 3.2]{BelHelVer01}. We introduce the following definition.

\begin{definition} \label{a26}
A set $E \subset \R^{d}$ is 1-null if the only function $f \in H^{1} ( \R^{d} )$ such that $\supp ( f ) \subset E$ is the zero function.
\end{definition}

\renewcommand{\labelitemi}{$\cdot$}

This terminology comes from \cite{HewMoi17}, where these sets are studied. In fact, this definition is very closed to the older notion of {\it set of uniqueness for $H^{1}$}, see \cite[Section 11.3]{AdaHed96} (and also \cite[Section 14.4]{Maz11}). These two notions coincide when $E$ is a closed set (which will be the case here), see \cite[Proposition 3.17]{HewMoi17}. We refer to \cite[Section 2]{HewMoi17} (and the references hereby) for a more detailed approach, including a characterization by the capacity of such sets and we just give below a list of examples illustrating this definition:
\begin{itemize}
\item If the interior of $E$ is not empty, $E$ is not 1-null.
\item The converse is true when $d = 1$, because functions in $H^{1} ( \R )$ are continuous.
\item If the Lebesgue measure of $E$ is zero, then $E$ is 1-null. 
\item For all $d \geq 2$, a compact set $E \subset \R^{d}$ of positive Lebesgue measure, which has empty interior and is not 1-null, is constructed explicitly in \cite[Theorem 3]{Pol72}.
\end{itemize}
This notion is related to our spectral problem through the following result.

\begin{proposition}[Behavior of the first eigenvalue]\sl
The set $V^{- 1} ( 0 )$ is $1$-null if and only if
\begin{equation*}
\lim_{h \to 0} \frac{\lambda_{1} ( P )}{h^{2}} = + \infty .
\end{equation*}
\end{proposition}

\begin{proof}
Assume that there exist $C > 0$ and a sequence $( h_{n} )_{n \in \N} \searrow 0$ such that
\begin{equation*}
\forall n \in \N , \qquad \lambda_{1} ( P ) \leq C h_{n}^{2} .
\end{equation*}
We will show that $V^{- 1} ( 0 )$ is not $1$-null. Let $u_{n}$ be an eigenfunction associated with $\lambda_{1} ( P )$ such that $\Vert u_{n} \Vert_{L^{2} ( \R^{d} )} = 1$. In particular,
\begin{equation*}
\forall n \in \N , \qquad \Vert \nabla u_{n} \Vert_{L^{2} ( \R^{d} )}^{2} + h_{n}^{- 2} \Vert V^{1 / 2} u_{n} \Vert_{L^{2} ( \R^{d} )}^{2} \leq C ,
\end{equation*}
and therefore $( u_{n} )_{n \in \N}$ is bounded in $H^{1} ( \R^{d} )$. Thus, there exists a subsequence $( v_{n} )_{n \in \N}$ which converges weakly to $v$ in $H^{1} ( \R^{d} )$. In particular, $( v_{n} )_{n \in \N}$ converges also to $v$ in $L^{2}_{\mathrm{l o c}} ( \R^{d} )$.

Let $\varepsilon \in ( 0 , \liminf_{\vert x \vert \to + \infty} V )$ and denote $K_{\varepsilon} : = V^{- 1} ( [ 0 , \varepsilon ] )$. Then 
\begin{equation*}
\varepsilon \int_{x \notin K_{\varepsilon}} \vert u_{n} ( x ) \vert ^{2} d x \leq \int_{x \notin K_{\varepsilon}} V ( x ) \vert u_{n} ( x ) \vert^{2} d x \leq C h_{n}^{2} ,
\end{equation*}
and we obtain 
\begin{equation*}
\int_{x \in K_{\varepsilon}} \vert u_{n} ( x ) \vert^{2} d x  = \Vert u_{n} \Vert_{L^{2} ( \R^{d} )}^{2} - \int_{x \notin K_{\varepsilon}} \vert u_{n} ( x ) \vert ^{2} d x \geq 1 - \frac{C}{\varepsilon} h_{n}^{2} .
\end{equation*}
Since $K_{\varepsilon}$ is compact, we get in the limit $n \to + \infty$
\begin{equation*}
\int_{x \in K_{\varepsilon}} \vert v ( x ) \vert^{2} d x \geq 1 .
\end{equation*}
Moreover, since $( v_{n} )_{n \in \N}$ converges weakly to $v$ in $H^{1} ( \R^{d} )$, it converges also weakly in $L^{2} ( \R^{d} )$ and $\Vert v \Vert_{L^{2} ( \R^{d} )} \leq 1$. Therefore, we deduce that $\Vert v \Vert_{L^{2} ( \R^{d} )} = 1$ and
\begin{equation*}
\supp ( v ) \subset \bigcap_{0 < \varepsilon \in \Q} K_{\varepsilon} = V^{- 1} ( 0 ) .
\end{equation*}
Since $v \in H^{1} ( \R^{d} )$, $V^{- 1} ( 0 )$ does not satisfy Definition \ref{a26}.

Conversely, assume that there exists $v \in H^{1} ( \R^{d} ) \setminus \{ 0 \}$ supported in $V^{- 1} ( 0 )$. Such a $v$ is in the form domain of $P$ and the maximin principle provides 
\begin{equation*}
\lambda_{1} ( P ) \leq \frac{\< P v , v \>}{\Vert v \Vert^{2}_{L^{2} ( \R^{d} )}} = h^{2} \frac{ \Vert \nabla v \Vert^{2}_{L^{2} ( \R^{d} )}}{\Vert v \Vert_{L^{2} ( \R^{d} )}^{2}} .
\end{equation*}
Therefore $\limsup_{h \to 0} \lambda_{1} ( P ) h^{- 2} < + \infty$.
\end{proof}

\section{Punctual wells in dimension $1$}

\subsection{A priori estimates of eigenvalues} \label{s3}

This part is devoted to the study of the low-lying eigenvalues in dimension $1$. We first give a general result for punctual wells. We consider $P$ satisfying \eqref{a25} with $d = 1$ and assume that $V^{- 1} ( 0 ) = \{ 0 \}$ and that $V$ is increasing for small positive $x$ and decreasing for small negative $x$. For $h > 0$ small enough, let $\delta_{\pm} ( h ) \in \pm ] 0 , + \infty [$ be the unique solutions of
\begin{equation} \label{a9}
V ( \delta_{-} ) = V ( \delta_{+} ) = \frac{h^{2}}{( \delta_{+} - \delta_{-} )^{2}} .
\end{equation}
Indeed, solving first $V ( \delta_{-} ) = V ( \delta_{+} )$ leads to $\delta_{-} = \delta_{-} ( \delta_{+} )$ with $\delta_{-} ( \cdot )$ continuous, negative and decreasing. Therefore, the continuous and increasing function $\delta_{+} \mapsto V ( \delta_{+} ) ( \delta_{+} - \delta_{-} ( \delta_{+} ) )^{2}$ meets $h^{2}$ at a unique point. According to the introduction, we have $I_{h} = ( \delta_{-} , \delta_{+} )$. For even potentials, \eqref{a9} becomes $\delta : = \delta_{+} = \vert \delta_{-} \vert$ with
\begin{equation} \label{a34}
4 \delta^{2} V ( \delta ) = h^{2} .
\end{equation}
The small eigenvalues of $P$ verify the following lower and upper bounds.

\begin{theorem}[Estimates in dimension $1$]\sl \label{a35}
Let $P$ be as before. Then, there exist constants $C_{k} , c_{k} > 0$ independent of $P$ such that
\begin{equation} \label{a28}
\frac{c_{k} h^{2}}{( \delta_{+} - \delta_{-} )^{2}} \leq \lambda_{k} ( P ) \leq \frac{C_{k} h^{2}}{( \delta_{+} - \delta_{-} )^{2}} ,
\end{equation}
for all $k \in \N^{*}$ and $h$ small enough.
\end{theorem}

In \eqref{a28}, we can always take $C_{k} = \pi^{2} k^{2} + 1$ but $c_{k}$ does not go to $+ \infty$ with $k$. In particular, we have no lower bound on the spectral gap $\lambda_{2} ( P ) - \lambda_{1} ( P )$. To prove this result, we make a scaling using $\delta_{\pm}$ and then compare the rescaled operator with some constant and simple operators.

\begin{proof}
Let us consider the unitary transformation on $L^{2} ( \R )$ given by
\begin{equation} \label{a22}
( U f ) ( x ) = ( \delta_{+} - \delta_{-} )^{- \frac{1}{2}} f \Big( \frac{x - \delta_{-}}{\delta_{+} - \delta_{-}} \Big) .
\end{equation}
Then, the rescaled operator is defined by
\begin{equation} \label{a8}
Q : = \frac{( \delta_{+} - \delta_{-} )^{2}}{h^{2}} U^{- 1} P U = - \Delta + W_{h} ( x ) ,
\end{equation}
where the potential $W_{h}$ satisfies
\begin{equation} \label{a29}
W_{h} ( x ) = \frac{( \delta_{+} - \delta_{-} )^{2}}{h^{2}} V \big( \delta_{-} + x ( \delta_{+} - \delta_{-} ) \big) .
\end{equation}
Using the monotonicity properties of $V$ and \eqref{a9}, we have
\begin{equation*}
W_{h} ( x ) \leq \frac{( \delta_{+} - \delta_{-} )^{2}}{h^{2}} \max ( V ( \delta_{-} ) , V ( \delta_{+} ) ) = 1 ,
\end{equation*}
for all $x \in [ 0 , 1 ]$. The same way,
\begin{equation*}
W_{h} ( x ) \geq \frac{( \delta_{+} - \delta_{-} )^{2}}{h^{2}} \min ( V ( \delta_{-} ) , V ( \delta_{+} ) ) = 1 ,
\end{equation*}
for all $x \in [ 0 , 1 ]^{C}$. Let $- \Delta_{I}^{D}$ denote the Dirichlet Laplacian on the open interval $I$. The previous estimates on $W_{h}$ imply
\begin{equation} \label{a32}
- \Delta + \one_{[ 0, 1 ]^{C}} \leq Q \leq - \Delta^{D}_{( 0 , 1 )} + 1 ,
\end{equation}
in the sense of form, see \cite[Section XIII.15]{ReSi78}. Then, the maximin principle yields
\begin{equation} \label{a33}
\lambda_{k} \big( - \Delta + \one_{[ 0, 1 ]^{C}}  \big) \leq \lambda_{k} ( Q ) \leq \lambda_{k} ( - \Delta^{D}_{( 0 , 1 )} ) + 1 .
\end{equation}
Eventually, the result follows from \eqref{a8}, \eqref{a33}, $\lambda_{k} ( - \Delta + \one_{[ 0, 1 ]^{C}} ) > 0$ and $\lambda_{k} ( - \Delta^{D}_{( 0 , 1 )} ) = \pi^{2} k^{2}$.
\end{proof}

\begin{example}\rm
Let $V$ be an even potential satisfying the assumptions of Theorem \ref{a35} and $V ( x ) = \vert \ln \vert x \vert \vert^{- \alpha}$ near $0$ with $\alpha > 0$. Then, \eqref{a34} gives $\delta \sim h \vert \ln h \vert^{\alpha / 2} / 2$ and $\lambda_{k} ( P )$ is of order $\vert \ln h \vert^{- \alpha}$.
\end{example}

\begin{example}\rm
Consider $V$ as in Theorem \ref{a35} with $V ( x ) \sim \vert x \vert$ for small negative $x$ and $V ( x ) \sim \vert \ln x \vert^{- 1}$ for small positive $x$. Then, \eqref{a9} implies $\vert \delta_{-} \vert \sim \vert \ln \delta_{+} \vert^{- 1}$, $\delta_{+} = \cO ( h^{\infty} )$ and $\vert \delta_{-} \vert \sim h^{2 / 3}$. As consequence, $\lambda_{k} ( P )$ is of order $h^{2 / 3}$. Such operator appears in \cite[Section 4.1]{DaOuRa15_01}.
\end{example}

\subsection{Asymptotic of eigenvalues for flat potentials} \label{s4}

We now consider $P$ as in \eqref{a25} with $d = 1$ and assume that $V^{- 1} ( 0 ) = \{ 0 \}$, $V$ is flat at $0$ and, for all $n \in \N$, the function
\begin{equation} \label{a27}
x \longmapsto \vert x \vert^{- n} V ( x ) ,
\end{equation}
is increasing for small positive $x$ and decreasing for small negative $x$. For smooth potentials, this hypothesis is equivalent to $x V^{\prime} ( x ) / V ( x ) \to + \infty$ as $x \to 0$. As before, we define $\delta_{\pm} ( h )$ by \eqref{a9}.


\begin{theorem}[Spectral asymptotic for flat potentials]\sl \label{a10}
Let $P$ satisfy the previous assumptions. In the limit $h \to 0$, we have
\begin{equation*}
\lambda_{k} ( P ) \sim \frac{\pi^{2} h^{2} k^{2}}{( \delta_{+} - \delta_{-} )^{2}} ,
\end{equation*}
for all $k \in \N^{*}$.
\end{theorem}

From \eqref{a9}, we have $h^{2} \lesssim ( \delta_{+} - \delta_{-} )^{2 + n}$ for all $n \in \N$ and therefore
\begin{equation*}
\forall \alpha > 0 , \qquad h^{\alpha} \ll \delta_{+} - \delta_{-} \ll 1 ,
\end{equation*}
as $h \to 0$. As a consequence, $\lambda_{k} ( P )$ goes to $0$ faster that any power of $h$ less than 2:
\begin{equation}
\forall \nu < 2 , \qquad h^{2} \ll \lambda_{k} ( P ) \ll h^{\nu} .
\end{equation}
The proof of this theorem uses the strategy of the one of Theorem \ref{a35}. We remark that, after an appropriate scaling, the operator $P$ looks like $- \Delta + R ( h ) \one_{[ 0 , 1 ]^{C}}$ with $R ( h ) \to + \infty$ as $h \to 0$. We conclude applying results on Schr\"{o}dinger operators with large coupling constant. The proof gives also an estimate of the remainder term in Theorem \ref{a10} (see Lemma \ref{a19}).

\begin{proof}
We begin with a lemma showing a difference between flat potentials and those having a non-zero Taylor expansion.

\begin{lemma}\sl \label{a18}
For all $0 < \varepsilon < 1$, we have
\begin{equation*}
\lim_{\fract{\delta \to 0}{\delta \neq 0}} \frac{V ( \varepsilon \delta )}{V ( \delta )} = 0 .
\end{equation*}
\end{lemma}

\begin{proof}[Proof of Lemma \ref{a18}]
We only consider the case $\delta > 0$, since the negative $\delta$ can be treated similarly. For all $n \in \N$, there exists $\delta_{n} > 0$ such that
\begin{equation*}
\forall \delta \in ( 0 , \delta_{n} ) , \qquad ( \varepsilon \delta )^{- n} V ( \varepsilon \delta ) \leq \delta^{- n} V ( \delta ) ,
\end{equation*}
from \eqref{a27}. Thus, for all $0 < \delta < \delta_{n}$,
\begin{equation} \label{a1}
0 \leq\frac{V ( \varepsilon \delta )}{V ( \delta )} \leq \varepsilon^{n} .
\end{equation}
Since $n$ can be chosen arbitrarily large, this implies the lemma.
\end{proof}

We now apply the unitary transformation \eqref{a22}. The rescaled operator $Q = - \Delta + W_{h}$ is defined in \eqref{a8}--\eqref{a29}. Roughly speaking, $W_{h}$ is very small in $[ 0 , 1 ]$ and very large outside this interval. More precisely,

\begin{lemma}\sl \label{a3}
For all $0 < \varepsilon < 1 / 2$, there exist two functions $m ( h ) , M ( h )$ with $m ( h ) \to 0$ and $M ( h ) \to + \infty$ as $h$ goes to $0$ such that
\begin{equation*}
\begin{aligned}
&0 \leq W_{h} ( x ) \leq m ( h ) &&\text{ for all } x \in [ \varepsilon , 1 - \varepsilon] ,  \\
&W_{h} ( x ) \geq M ( h ) &&\text{ for all } x \in [ - \varepsilon , 1 + \varepsilon]^{C} .
\end{aligned}
\end{equation*}
\end{lemma}

\begin{proof}[Proof of Lemma \ref{a3}]
From the monotonicity properties of $W_{h}$ near $[0 , 1]$, it is enough to verify the lemma at $\pm \varepsilon$ and $1 \pm \varepsilon$. Using \eqref{a9} and $\pm \delta_{\pm} > 0$, we have
\begin{equation*}
W_{h} ( 1 - \varepsilon ) = \frac{( \delta_{+} - \delta_{-} )^{2}}{h^{2}} V \big( ( 1 - \varepsilon ) \delta_{+} + \varepsilon \delta_{-} \big) \leq \max \bigg( \frac{V \big( ( 1 - \varepsilon ) \delta_{+} \big)}{V ( \delta_{+} )} , \frac{V ( \varepsilon \delta_{-} )}{V ( \delta_{-} )} \bigg) = : m_{+} ( h ) .
\end{equation*}
Similar computations show that $W_{h} ( \varepsilon ) \leq m_{-} ( h )$ where 
\begin{equation*}
m_{-} ( h ) = \max \bigg( \frac{V \big( ( 1 - \varepsilon ) \delta_{-} \big)}{V ( \delta_{-} )} , \frac{V ( \varepsilon \delta_{+} )}{V ( \delta_{+} )} \bigg) .
\end{equation*}
Setting $m ( h ) = \max(m_{\pm} ( h ) )$, the first part of the lemma follows from Lemma \ref{a18}.

More easily, \eqref{a9}, $\delta_{-} < 0$ and the monotonicity properties of $V$ yield
\begin{equation*}
W_{h} ( 1 + \varepsilon ) = \frac{( \delta_{+} - \delta_{-} )^{2}}{h^{2}} V \big( ( 1 + \varepsilon ) \delta_{+} - \varepsilon \delta_{-} \big) \geq \frac{V \big( ( 1 + \varepsilon ) \delta_{+} \big)}{V(\delta_{+})} :=M_{+}( h ) ,
\end{equation*}
and similarly 
\begin{equation*}
W_{h} ( - \varepsilon ) \geq \frac{V \big( ( 1 + \varepsilon ) \delta_{-} \big)}{V ( \delta_{-} )} = : M_{-} ( h ) .
\end{equation*}
We set $M(h)=\min(M_{\pm}(h))$ and we deduce the second part of the lemma from Lemma \ref{a18}.
\end{proof}

We apply the estimates on the potential $W_{h}$ to surround the eigenvalues of $Q$ as in \eqref{a33}.

\begin{lemma}\sl \label{a19}
For all $k \in \N^{*}$, there exists $C > 0$ such that, for all $0 < \varepsilon < 1 / 2$, there exists $h_{\varepsilon} > 0$ with
\begin{equation} \label{a24}
\forall h \in ( 0 , h_{\varepsilon} ) , \qquad \frac{\pi^{2} k^{2}}{( 1 + 2 \varepsilon )^{2}} - \frac{C}{M ( h )^{1 / 2}} \leq \lambda_{k} ( Q ) \leq \frac{\pi^{2} k^{2}}{( 1 - 2\varepsilon )^{2}} + m ( h ) .
\end{equation}
\end{lemma}

\begin{proof}
From Lemma \ref{a3}, we have
\begin{equation} \label{a21}
- \Delta + M ( h ) \one_{[ - \varepsilon , 1 + \varepsilon]^{C}} \leq Q \leq - \Delta^{D}_{( \varepsilon , 1 - \varepsilon )} + m ( h ) ,
\end{equation}
in the sense of form. Then, the maximin principle gives
\begin{equation} \label{a4}
\lambda_{k} \big( - \Delta + M ( h ) \one_{[ - \varepsilon , 1 + \varepsilon]^{C}} \big) \leq \lambda_{k} ( Q ) \leq \lambda_{k} ( - \Delta^{D}_{( \varepsilon , 1 - \varepsilon )} ) + m ( h ) .
\end{equation}
On the one hand, applying the translation $x \mapsto x + \varepsilon$ and the scaling $x \mapsto ( 1 + 2 \varepsilon )^{- 1} x$, we see that the operator $- \Delta + M ( h ) \one_{[ - \varepsilon , 1 + \varepsilon]^{C}}$ is unitarily equivalent to $( 1 + 2 \varepsilon )^{- 2} ( - \Delta + ( 1 + 2 \varepsilon )^{2} M ( h ) \one_{[ 0 , 1 ]^{C}} )$. As $h \to 0$, this operator enters the theory of Schr\"{o}dinger operators with large coupling constant, see \cite{AshHar82,BrCa02_01} and also \cite[Problem 25]{Flu99} for explicit computations. In particular, for fix $k \in \N^{*}$, \cite[Theorem 3.6]{AshHar82} provides
\begin{equation*}
\lambda_{k} \big( - \Delta + ( 1 + 2 \varepsilon )^{2} M ( h ) \one_{[ 0 , 1]^{C}} \big) = \lambda_{k} ( - \Delta^{D}_{( 0 , 1 )} ) + \cO \bigg( \frac{1}{( 1 + 2\varepsilon ) M ( h )^{1 / 2}} \bigg) ,
\end{equation*}
as $h \to 0$, and therefore 
\begin{equation*}
\lambda_{k} \big( - \Delta + M ( h ) \one_{[ - \varepsilon , 1 + \varepsilon]^{C}} \big) \geq \frac{\pi^{2} k^{2}}{( 1 + 2 \varepsilon )^{2}} - \frac{C}{M ( h )^{1 / 2}} .
\end{equation*}
On the other hand, we have $\lambda_{k} ( - \Delta^{D}_{( \varepsilon , 1 - \varepsilon )}) = \pi^{2} k^{2} ( 1 - 2 \varepsilon )^{-2}$ and we get the lemma from \eqref{a4}.
\end{proof}
Finally, Theorem \ref{a10} is a direct consequence of \eqref{a8}, Lemma \ref{a3} and Lemma \ref{a19}.
\end{proof}

We now treat some examples of potentials. For simplicity, we will consider even potentials, so that $\delta_{+} = \vert \delta_{-} \vert = : \delta$ and the functions $m ( h )$ and $M ( h )$ of Lemma \ref{a3} are given by 
\begin{equation*}
m ( h ) = \frac{V \big( ( 1 - \varepsilon ) \delta \big)}{V ( \delta )} \qquad \text{and} \qquad M ( h ) = \frac{V \big( ( 1 + \varepsilon ) \delta \big)}{V ( \delta )} ,
\end{equation*}
for $0 < \varepsilon < 1 / 2$. Note that all these examples are of the form $V ( x ) = x^{\omega ( x )}$ with $\omega ( x ) \to + \infty$ as $x \to 0$, which is the generic form of flat potential.

\begin{example}\rm \label{a31}
$V ( x ) = e^{- \vert x \vert^{- \alpha}}$ with $\alpha > 0$. This potential satisfies the assumptions of Theorem \ref{a10} and $\delta$ is given by solving $4 \delta^{2} e^{- \delta^{- \alpha}} = h^{2}$. Taking the logarithm, we get $\delta ( h ) \sim ( 2 \vert \ln h \vert )^{- 1 / \alpha}$. We then optimize the upper bound in \eqref{a24} by solving $\varepsilon = m ( h )$, that is $\varepsilon = e^{- \delta^{- \alpha} ( ( 1 - \varepsilon )^{- \alpha} - 1 ) }$. Quick computations show that the solution $\varepsilon_{+}$ verifies $\varepsilon_{+} ( \delta ) \sim \delta^{\alpha} \vert \ln \delta \vert$ as $h \to 0$. Similarly, optimizing the lower bound in \eqref{a24} leads to $\varepsilon = \varepsilon_{-} ( \delta ) \sim 2 \delta^{\alpha} \vert \ln \delta \vert$. Eventually, Lemma \ref{a19} provides as $h \to 0$
\begin{equation} \label{a30}
\lambda_{k} ( P ) = h^{2} \vert \ln h \vert^{2 / \alpha} \bigg( 2^{- 2 + 2 / \alpha} \pi^{2} k^{2} + \cO \bigg( \frac{\vert \ln \vert \ln h \vert \vert}{\vert \ln h \vert} \bigg) \bigg) .
\end{equation}
\end{example}

\begin{example}\rm
$V ( x ) = e^{- \vert \ln \vert x \vert \vert^{2}}$. The hypotheses of Theorem \ref{a10} hold and $\delta = e^{1 - \sqrt{1 + 2 \vert \ln h / 2 \vert}}$ in that case. Optimization of the remainders is possible as in Example \ref{a31}.
\end{example}

\begin{example}\rm
$V ( x ) = e^{- \vert x \vert^{- 4} - \vert x \vert^{- 2} ( 1 + \sin ( \vert x \vert^{- 36} ) )}$. This potential is continuous, non-negative, flat at $0$ with $V^{- 1} ( 0 ) = \{ 0 \}$ and $\liminf_{\vert x \vert \to + \infty} V ( x ) > 0$. Nevertheless, \eqref{a27} does not hold since $V$ itself is not increasing for small positive $x$. Thus, Theorem \ref{a10} can not be applied here.
\end{example}

The proof of Theorem \ref{a10} gives as a byproduct the description of the eigenvectors of $P$.

\begin{proposition}[Asymptotic of eigenvectors]\sl \label{a20}
Under the assumptions of Theorem \ref{a10}, there exists a normalized eigenvector $u_{k}$ of $P$ associated to the eigenvalue $\lambda_{k} ( P )$ with $k \in \N^{*}$ such that
\begin{equation} \label{a36}
u_{k} ( x ) = \sqrt{\frac{2}{\delta_{+} - \delta_{-}}} \sin \Big( \pi k \frac{x - \delta_{-}}{\delta_{+} - \delta_{-}} \Big) \one_{[ \delta_{-} , \delta_{+} ]} + o_{L^{2} ( \R )} ( 1 ) ,
\end{equation}
in the limit $h \to 0$.
\end{proposition}

The functions in the right hand side of the last equation are normalized in $L^{2} ( \R )$, belong to the form domain of $P$ but not in its domain. In fact, they form a basis of eigenvectors of $- h^{2} \Delta_{( \delta_{-} , \delta_{+} )}^{D}$. This result is then in agreement with the intuition that $P$ behaves like this operator at low energy.

\begin{proof}
Let $v_{k}$ be the function in the right hand side of \eqref{a36}. We show this proposition by induction over $k$. For $k = 0$, there is nothing to prove. Assume that this property holds true until $k - 1 \in \N$. Using the unitary transform \eqref{a22}, we note $U_{\ell} : = U^{- 1} u_{\ell}$ for $1 \leq \ell \leq k - 1$,
\begin{equation*}
V_{\ell} : = U^{- 1} v_{\ell} = \sqrt{2} \sin ( \pi \ell x ) \one_{[ 0 , 1 ]} ( x ) \qquad \text{and} \qquad V_{k}^{\varepsilon} : = \sqrt{\frac{2}{1 - 2 \varepsilon}} \sin \Big( \frac{\pi k ( x - \varepsilon )}{1 - 2 \varepsilon} \Big) \one_{[ \varepsilon , 1 - \varepsilon ]} ( x ) ,
\end{equation*}
for $1 \leq \ell \leq k$. In particular, $( U_{\ell} )_{1 \leq \ell \leq k - 1}$ (resp. $( V_{\ell} )_{1 \leq \ell \leq k}$) is an orthonormal basis of the eigenspace associated to the $k - 1$ (resp. $k$) first eigenvalues of $Q$ (resp. $- \Delta^{D}_{(0 , 1 )}$). Otherwise, $V_{k}^{\varepsilon}$ is a normalized eigenvector of $- \Delta_{( \varepsilon , 1 - \varepsilon )}^{D}$ associated to the eigenvalue $\pi^{2} k^{2} ( 1 - 2 \varepsilon )^{- 2}$ and $V_{k}^{\varepsilon} = V_{k} + o_{\varepsilon \to 0} ( 1 )$. Let us decompose $V_{k}^{\varepsilon}$ using
\begin{equation*}
V_{k}^{-} = \one_{[ 0 , \lambda_{k - 1} ( Q ) ]} ( Q ) V_{k}^{\varepsilon} , \qquad V_{k}^{0} = \one_{\{ \lambda_{k} ( Q ) \}} ( Q ) V_{k}^{\varepsilon} , \qquad V_{k}^{+} = \one_{[ \lambda_{k + 1} ( Q ) , + \infty )} ( Q ) V_{k}^{\varepsilon} .
\end{equation*}
Of course, $V_{k}^{\varepsilon} = V_{k}^{-} + V_{k}^{0} + V_{k}^{+}$ and $1 = \Vert V_{k}^{-} \Vert^{2} + \Vert V_{k}^{0} \Vert^{2} + \Vert V_{k}^{+} \Vert^{2}$. Moreover, by the induction hypothesis and the orthogonality of the $v_{k}$,
\begin{equation} \label{a23}
V_{k}^{-} = o_{\varepsilon \to 0} ( 1 ) + o_{h \to 0} ( 1 ) .
\end{equation}
Using that $V_{k}^{\varepsilon}$ is in the form domain of $- \Delta_{( \varepsilon , 1 - \varepsilon )}^{D}$, \eqref{a21} gives
\begin{equation*}
\big\< Q V_{k}^{\varepsilon} , V_{k}^{\varepsilon} \big\> \leq \big\< - \Delta_{( \varepsilon , 1 - \varepsilon )}^{D} V_{k}^{\varepsilon} , V_{k}^{\varepsilon} \big\> + \big\< m ( h ) V_{k}^{\varepsilon} , V_{k}^{\varepsilon} \big\> = \pi^{2} k^{2} ( 1 - 2 \varepsilon)^{- 2} + o_{h \to 0}^{\varepsilon} ( 1 ) ,
\end{equation*}
where $o_{h \to 0}^{\varepsilon} ( 1 )$ denotes a function which goes to $0$ as $h$ goes to $0$ for $\varepsilon$ fixed. From \eqref{a23} and $\lambda_{\ell} ( Q ) \sim \pi^{2} \ell^{2}$, we obtain
\begin{align*}
\big\< Q V_{k}^{\varepsilon} , V_{k}^{\varepsilon} \big\> &= \big\< Q V_{k}^{-} , V_{k}^{-} \big\> + \big\< Q V_{k}^{0} , V_{k}^{0} \big\> + \big\< Q V_{k}^{+} , V_{k}^{+} \big\>  \\
&\geq \pi^{2} k^{2} \Vert V_{k}^{0} \Vert^{2} + \pi^{2} ( k + 1 )^{2} \Vert V_{k}^{+} \Vert^{2} + o_{\varepsilon \to 0} ( 1 ) + o_{h \to 0}^{\varepsilon} ( 1 ) .
\end{align*}
The two last inequalities and the properties of $V_{k}^{\varepsilon}$ yield
\begin{equation*}
\pi^{2} k^{2} \Vert V_{k}^{0} \Vert^{2} + \pi^{2} ( k + 1 )^{2} \Vert V_{k}^{+} \Vert^{2} \leq \pi^{2} k^{2} \Vert V_{k}^{0} \Vert^{2} + \pi^{2} k^{2} \Vert V_{k}^{+} \Vert^{2} + o_{\varepsilon \to 0} ( 1 ) + o_{h \to 0}^{\varepsilon} ( 1 ) ,
\end{equation*}
and then $\Vert V_{k}^{+} \Vert = o_{\varepsilon \to 0} ( 1 ) + o_{h \to 0}^{\varepsilon} ( 1 )$. Summing up,
\begin{align*}
V_{k} &= V_{k}^{\varepsilon} + o_{\varepsilon \to 0} ( 1 ) = \one_{\{ \lambda_{k} ( Q ) \}} ( Q ) V_{k}^{\varepsilon} + o_{\varepsilon \to 0} ( 1 ) + o_{h \to 0}^{\varepsilon} ( 1 ) \\
&= \one_{\{ \lambda_{k} ( Q ) \}} ( Q ) V_{k} + o_{\varepsilon \to 0} ( 1 ) + o_{h \to 0}^{\varepsilon} ( 1 ) .
\end{align*}
Coming back to the original variables and using that the function $v_{k}$ is independent of $\varepsilon$, we deduce $v_{k} = \one_{\{ \lambda_{k} ( P ) \}} ( P ) v_{k} + o_{h \to 0} ( 1 )$ which implies the induction hypothesis for $k$.
\end{proof}

\section{A generalization in higher dimensions} \label{s5}

Here, we study operators $P$ as in \eqref{a25} on $\R^{d}$ with $d \geq 1$ where the potential $V$ can be written
\begin{equation} \label{a13}
V ( x ) = V_{0} \left( \vert x \vert \theta ( \widehat{x} ) \right) ,
\end{equation}
and $\widehat{x} = x \vert x \vert^{- 1} \in \S^{d - 1}$ is the angle of $x$. We assume that $V_{0} \in C^{0} ( \R_{+} ; \R_{+} )$ satisfies $V^{- 1}_{0} ( 0 ) = \{ 0 \}$, $\liminf_{x \to + \infty} V_{0} ( x ) > 0$, $V_{0}$ is flat at $0$ and $x \mapsto \vert x \vert^{- n} V_{0} ( x )$ is increasing for all $n \in \N$ and small positive $x$. We also suppose that $\theta$ belongs to $C^{0} ( \S^{d - 1} ; \R_{+}^{*} )$. Let $\Omega$ denote the star-shaped, bounded open set defined by
\begin{equation*}
\Omega = \{ x \in \R^{d} ; \ \vert x \vert < 1 / \theta ( \widehat{x} ) \} .
\end{equation*}
Mimicking \eqref{a9}, let $\delta ( h ) \in ] 0 , + \infty [$ be the unique solution of
\begin{equation} \label{a11}
\delta^{2} V_{0} ( \delta ) = h^{2} ,
\end{equation}
for $h$ small enough.

\begin{theorem}[Spectral asymptotic in dimension $d$]\sl \label{a12}
Let $P$ satisfy the previous assumptions. In the limit $h \to 0$, we have
\begin{equation*}
\lambda_{k} ( P ) \sim \frac{h^{2}}{\delta^{2}} \lambda_{k} ( - \Delta_{\Omega}^{D} ) ,
\end{equation*}
for all $k \in \N^{*}$, where $- \Delta_{\Omega}^{D}$ is the Dirichlet Laplacian on $\Omega$.
\end{theorem}

Note that the eigenvalues of $- \Delta_{\Omega}^{D}$ are positive. Radial potentials can be considered taking $\theta = 1$. In that case, $\Omega$ is the unit ball $B ( 0 , 1 )$. On the other hand, Theorems \ref{a10} and \ref{a12} provide similar results in dimension $d = 1$ under the present assumptions. Indeed, direct computations show $\delta ( h ) \sim \pm \theta ( \pm 1 ) \delta_{\pm} ( h )$ and $\Omega = ( 1 / \theta ( - 1 ) , 1 / \theta ( + 1 ) )$. Finally, if we assume $V ( x ) = V_{0} ( \vert x \vert ) \theta ( \widehat{x} )$ instead of \eqref{a13}, one can verify that $\lambda_{k} ( P ) \sim h^{2} \delta^{- 2} \lambda_{k} ( - \Delta_{B ( 0 , 1 )}^{D} )$. In other words, $\theta$ plays no role and the geometry (given by $\Omega$) disappears.

\begin{proof}
The proof is similar to the one of Theorem \ref{a10}. Let us consider the unitary transformation on $L^{2} ( \R^{d} )$ given by
\begin{equation*}
( U f ) ( x ) = \delta^{- \frac{d}{2}} f \Big( \frac{x}{\delta} \Big) .
\end{equation*}
As in \eqref{a8}, the rescaled operator is defined by
\begin{equation} \label{a6}
Q : = \frac{\delta^{2}}{h^{2}} U^{- 1} P U = - \Delta + W_{h} ( x ) ,
\end{equation}
where the potential $W_{h}$ satisfies
\begin{equation*}
W_{h} ( x ) = \frac{\delta^{2}}{h^{2}} V ( \delta x ) = \frac{\delta^{2}}{h^{2}} V_{0} \big( \delta \vert x \vert \theta ( \widehat{x} ) \big) .
\end{equation*}
As in Lemma \ref{a19}, for all $0 < \varepsilon < 1 / 2$, there exist two functions $m ( h ) , M ( h )$ with $m ( h ) \to 0$ and $M ( h ) \to + \infty$ as $h$ goes to $0$ such that
\begin{equation} \label{a14}
\begin{aligned}
&0 \leq W_{h} ( x ) \leq m ( h ) &&\text{ for all } x \in ( 1 - \varepsilon ) \Omega ,  \\
&W_{h} ( x ) \geq M ( h ) &&\text{ for all } x \notin ( 1 + \varepsilon ) \Omega  .
\end{aligned}
\end{equation}
Indeed, for $x \in ( 1 - \varepsilon ) \Omega$, we have $\delta \vert x \vert \theta ( \widehat{x} ) \leq ( 1 - \varepsilon ) \delta$. Then, Lemma \ref{a18} and \eqref{a11} give
\begin{equation*}
W_{h} ( x ) \leq \frac{\delta^{2}}{h^{2}} V_{0} ( ( 1 - \varepsilon ) \delta ) \leq \frac{\delta^{2}}{h^{2}} V_{0} ( \delta ) m ( h ) = m ( h ) .
\end{equation*}
The same way, for $x \notin ( 1 + \varepsilon ) \Omega$, we have $\delta \vert x \vert \theta ( \widehat{x} ) \geq ( 1 + \varepsilon ) \delta$. Then, Lemma \ref{a18} and \eqref{a11} give
\begin{equation*}
W_{h} ( x ) \geq \frac{\delta^{2}}{h^{2}} V_{0} ( ( 1 + \varepsilon ) \delta ) \geq \frac{\delta^{2}}{h^{2}} V_{0} ( \delta ) M ( h ) = M ( h ) .
\end{equation*}
From \eqref{a14}, we deduce as in \eqref{a21} that
\begin{equation*}
- \Delta_{\R^{d}} + M ( h ) \one_{( 1 + \varepsilon ) \Omega^{C}} \leq Q \leq - \Delta_{( 1 - \varepsilon ) \Omega}^{D} + m ( h ) ,
\end{equation*}
in the form sense. Then, the maximin principle yields
\begin{equation} \label{a15}
\lambda_{k} \big( - \Delta_{\R^{d}} + M ( h ) \one_{( 1 + \varepsilon ) \Omega^{C}} \big) \leq \lambda_{k} ( Q ) \leq \lambda_{k} \big( - \Delta_{( 1 - \varepsilon ) \Omega}^{D} + m ( h ) \big) ,
\end{equation}
for all $k \in \N^{*}$.

By scaling invariance, we have
\begin{equation} \label{a17}
\lambda_{k} \big( - \Delta_{( 1 - \varepsilon ) \Omega}^{D} + m ( h ) \big) = ( 1 - \varepsilon )^{- 2} \lambda_{k} ( - \Delta_{\Omega}^{D} ) + o_{h \to 0}^{\varepsilon} ( 1 ) ,
\end{equation}
where $o_{h \to 0}^{\varepsilon} ( 1 )$ denotes a function which goes to $0$ as $h$ goes to $0$ for $\varepsilon$ fixed. On the other hand, by the theory of large coupling constant (see \cite{Si78_01}, \cite{GeGuHoKlSaSiVo88_01} or \cite{De91_01}), the spectrum of $- \Delta_{\R^{d}} + M \one_{( 1 + \varepsilon ) \Omega^{C}}$ converges to the one of $- \Delta_{( 1 + \varepsilon ) \Omega}^{D}$ as $M$ goes to $+ \infty$. Thus,
\begin{align}
\lambda_{k} \big( - \Delta_{\R^{d}} + M ( h ) \one_{( 1 + \varepsilon ) \Omega^{C}} \big) &= \lambda_{k} \big( - \Delta_{( 1 + \varepsilon ) \Omega}^{D} \big) + o_{h \to 0}^{\varepsilon} ( 1 )   \nonumber \\
&= ( 1 + \varepsilon )^{- 2} \lambda_{k} ( - \Delta_{\Omega}^{D} ) + o_{h \to 0}^{\varepsilon} ( 1 ) . \label{a16}
\end{align}
Combining \eqref{a15} with \eqref{a6}, \eqref{a17} and \eqref{a16}, we get
\begin{equation*}
( 1 + \varepsilon )^{- 2} \lambda_{k} ( - \Delta_{\Omega}^{D} ) + o_{h \to 0}^{\varepsilon} ( 1 ) \leq \frac{\delta^{2}}{h^{2}} \lambda_{k} ( P ) \leq ( 1 - \varepsilon )^{- 2} \lambda_{k} ( - \Delta_{\Omega}^{D} ) + o_{h \to 0}^{\varepsilon} ( 1 ).
\end{equation*}
Letting $\varepsilon$ goes to $0$, this inequality and $\lambda_{k} ( - \Delta_{\Omega}^{D} ) > 0$ imply Theorem \ref{a12}.
\end{proof}

\bibliographystyle{amsplain}
\providecommand{\MR}{\relax\ifhmode\unskip\space\fi MR }
\providecommand{\MRhref}[2]{%
  \href{http://www.ams.org/mathscinet-getitem?mr=#1}{#2}
}
\providecommand{\href}[2]{#2}

\end{document}